\documentclass[12pt,reqno]{amsart}
\usepackage[notref, notcite]{}
\usepackage{amscd} 
\usepackage[hyphens]{xurl}
\usepackage{mathrsfs}
\usepackage{amsthm}
\usepackage{amsmath}
\usepackage{amssymb}
\usepackage{amsfonts}
\usepackage{pb-diagram} 
\usepackage[all]{xy}
\usepackage{graphicx}
\usepackage{latexsym}
\usepackage{comment}
\usepackage{ascmac}
\usepackage{caption}
\usepackage{mathtools}
\usepackage{color}
\numberwithin{equation}{section}

\makeatletter

    \newcommand\contFrac{\@ifstar{\@contFracStar}{\@contFracNoStar}}

   \def\singleContFrac#1#2{%
        \begin{array}{@{}c@{}}%
            \multicolumn{1}{c|}{#1}%
            \\%
            \hline%
           \multicolumn{1}{|c}{#2}%
        \end{array}%
   }

    \def\@contFracNoStar#1{%
        \mathchoice{
            \@contFracNoStarDisplay@#1//\@nil%
        }{
            \@contFracNoStarInline@#1//\@nil%
        }{
            \@contFracNoStarInline@#1//\@nil%
        }{
            \@contFracNoStarInline@#1//\@nil%
        }%
    }

    \def\@contFracNoStarDisplay@#1//#2\@nil{%
        \@ifmtarg{#2}{%
            #1%
        }{%
            #1+\cfrac{1}{\@contFracNoStarDisplay@#2\@nil}%
        }%
    }

        \def\@contFracNoStarInline@#1//#2\@nil{%
            \@ifmtarg{#2}{%
                #1%
            }{%
                #1 \@@contFracNoStarInline@@#2\@nil%
            }%
        }
        \def\@@contFracNoStarInline@@#1//#2\@nil{%
            \@ifmtarg{#2}{%
                + \singleContFrac{1}{#1}%
            }{%
                + \singleContFrac{1}{#1} \@@contFracNoStarInline@@#2\@nil%
            }%
        }

    \def\@contFracStar#1{%
        \mathchoice{
            \@contFracStarDisplay@#1////\@nil%
        }{
            \@contFracStarInline@#1//\@nil%
        }{
            \@contFracStarInline@#1//\@nil%
        }{
            \@contFracStarInline@#1//\@nil%
        }%
    }

    \def\@contFracStarDisplay@#1//#2//#3\@nil{%
        \@ifmtarg{#2}{%
            #1%
        }{%
            #1 + \cfrac{#2}{\@contFracStarDisplay@#3\@nil}%
        }%
    }

        \def\@contFracStarInline@#1//#2\@nil{%
            \@ifmtarg{#2}{%
                #1%
            }{%
                #1 \@@contFracStarInline@@#2\@nil%
            }%
        }
        \def\@@contFracStarInline@@#1//#2//#3\@nil{%
            \@ifmtarg{#3}{%
                + \singleContFrac{#1}{#2}%
            }{%
                + \singleContFrac{#1}{#2} \@@contFracStarInline@@#3\@nil%
            }%
        }

\theoremstyle{plain}
\newtheorem*{syuA}{Main Theorem A}
\newtheorem*{syuB}{Main Theorem B}
\newtheorem{thm}{Theorem}[section]
\newtheorem{lem}[thm]{Lemma}
\newtheorem{cor}[thm]{Corollary}

\newtheorem{pro}[thm]{Proposition}
\theoremstyle{definition}
\newtheorem{df}[thm]{Definition}

\newtheorem{rem}[thm]{Remark}

\newtheorem*{prf*}{Proof}

\newtheorem*{pf*}{}
\newtheorem*{lem*}{LemmaA}
\newtheorem*{lm*}{LemmaB}
\newtheorem*{stra*}{Strategy for the proof of main result A}

\DeclareMathOperator*{\ZZ}{\mathbb{Z}}
\DeclareMathOperator*{\RR}{\mathbb{R}}
\makeatletter
\@namedef{subjclassname@2020}{%
  \textup{2020} Mathematics Subject Classification}
\makeatother
\title[Topology of Slicing ST]{
Topology of slices through \\the Sierpi\'nski tetrahedron}

\author{Yuto Nakajima and Takayuki Watanabe
}

\date{}

\address{Faculty of Science and Engineering, Doshisha University, 1-3 Tatara Miyakodani, Kyotanabe-shi, Kyoto, 610-0394, Japan}
\email{yunakaji@mail.doshisha.ac.jp}
\address{Chubu University Academy of Emerging Sciences,
1200 Matsumoto-cho, Kasugai-shi, Aichi, 487-8501, Japan.} 
\email{takawatanabe@fsc.chubu.ac.jp}

\begin{document}

\begin{abstract}
We investigate slices of the Sierpi\'nski tetrahedron from a topological viewpoint.
For each $c\in[0,1]$, we study the \v{C}ech (co)homology group of the slice at height $c$.
We show that the topology of the slice exhibits a sharp dichotomy.
If $c$ is a dyadic rational, then the slice has finitely many connected components, infinite first \v{C}ech homology, and trivial higher homology.
If $c$ is not a dyadic rational, then the slice is totally disconnected and all positive-degree \v{C}ech homology groups vanish.

\end{abstract}

\maketitle

\section{Introduction}
\subsection{Background}
In the study of slicing problems for fractals, a central objective is to determine the dimension of the slices. General results of Marstrand-type slicing theorems can be found in \cite{Mar, Mat, Mat2}. Results concerning slices of sets generated by Moran constructions in cubic settings are presented in \cite{BFS, BP, Nak22, WWX, WX, WXX}. For slicing problems arising from intersections of two Cantor sets, see \cite{LX, Sh, Wu}. 

Our main aim in this paper is to investigate the topological properties of slices of fractals.
For limit sets of iterated function systems (IFSs), the celebrated work of Hata \cite{Hata} provided a criterion for determining whether the limit set is connected (see also \cite{Nak24, W24} for related results on the connectivity of fractals).
To capture more refined topological structures of limit sets, we consider \v{C}ech homology groups, whose definition will be recalled in the next section. In particular, the zeroth \v{C}ech homology group characterizes the connectivity of a given topological space.
In this direction, a pioneering work was carried out by Sumi \cite{Sumi09}. Inspired by his work, the authors \cite{NW} established a homological framework for a broad class of fractals and calculated the rank of \v{C}ech homology groups for sets arising from Moran constructions in cubic settings.

In this paper, we focus on   slices of the so-called Sierpi\'nski tetrahedron, which serves as a representative example of fractals. The Sierpi\'nski tetrahedron is generated by an IFS consisting of similitudes. However, slices of the set are not expected to be limit sets of IFSs in general, which gives rise to essential difficulties. 
Moreover, since the present setting is not covered by the examples treated in \cite{NW}, new ideas are required to address our slicing problems.

\subsection{Main results}
In this section we give 
rigorous definitions for our setting and present main results on them. We first define the Sierpi\'nski tetrahedron.
\begin{df}
\label{sier}
Let $v_0=(0, 0, 0), v_1=(0, 0, 1), v_2=(0, 1, 1),$ and $v_3=(1, 1, 1)$ be vertices of a tetrahedron in $\mathbb R^3$. 
For each $i \in \{0,1,2,3\}$, define a contracting map $f_i: \mathbb R^3\rightarrow \mathbb R^3$ by $f_i(x)=(x+v_i)/2$. Then there uniquely exists a non-empty compact subset $J$ of $\mathbb R^3$ such that $
J=\bigcup_{i=0}^3f_i(J)$
(see e.g. \cite{Fal, Hut}),
which is called the Sierpi\'nski tetrahedron.
\end{df}

For $x=(x_1, x_2, x_3)\in \mathbb{R}^3$, set $p_z(x)=x_3$. Let $c\in [0, 1]$ and consider the following level set $J_{c}$ of $J$: 
\begin{align*}
J_{c}=\{x\in J\ :\ p_z(x)=c\},
\end{align*}
which is the slice of $J$ at height $c$.


Let $\Delta$ be a topological space. 
For a collection $\mathfrak{U}$ of subsets of $\Delta$, let $N(\mathfrak{U})$ be the (abstract) simplicial complex, which is called the nerve,  
whose simplexes are finite non-empty subsets of $\mathfrak{U}$ with non-empty intersection.
Then for every $q\geq 0,$ the $q$th \v{C}ech homology group $\check{H}_q(\Delta)$ is defined as the inverse limit
$\check{H}_q(\Delta) = \varprojlim_{\mathfrak{U}} H_{q} (N(\mathfrak{U}))$, where $H_{q}(\mathcal{K})$ is the $q$th homology group of a simplicial complex $\mathcal{K}$ and 
where $\mathfrak{U}$ runs over all finite open coverings of $\Delta$ ordered by refinement. 
Similarly, the $q$th \v{C}ech cohomology group $\check{H}^q(\Delta) = \varinjlim_{\mathfrak{U}} H^{q} (N(\mathfrak{U}))$ is defined as the direct limit.

A number $c\in [0, 1]$ is called a dyadic rational if there exist $n\in \mathbb N$ and $k\in \{0, 1,\ldots, 2^n-1\}$ such that $c=k/2^n.$ 
The following is the first main result in this paper.
\begin{syuA}
\label{SyuA}
Let $c\in [0, 1].$ Then the following holds.
\begin{itemize}

\item[(a)] If $c$ is a dyadic rational, then $J_c$ is a finite disjoint union of copies of the so-called Sierpi\'nski gasket. 
Moreover, $\check{H}_0(J_c)$ and $\check{H}^0(J_c)$ are (isomorphic to) $\ZZ^r$ with some $r \geq 1$, ${\rm rank}\check{H}_1(J_c) = \infty = {\rm rank}\check{H}^1(J_c),$ and $\check{H}_q(J_c) =\check{H}^q(J_c)=0$ for all $q\geq 2.$

\item[(b)]If $c$ is a non-dyadic rational, then $J_c$ is totally disconnected and $\check{H}_q(J_c) = \check{H}^q(J_c)=0$ for all $q\geq 1.$
\end{itemize}

\end{syuA}
The proof of Main Theorem A is reduced to the theory of non-autonomous iterated function systems (see Definition~\ref{NIFS}), which generalize iterated function systems.
More precisely, for any $c\in [0, 1]$ the slice $J_c$ can be realized as the limit set of a non-autonomous IFS (see Lemma~\ref{SliceNIFS}). 
For any $c\in [0, 1]$ the associated non-autonomous IFS is determined by the binary expansion of $c$, defined as follows. 
\begin{df}
\label{bi}
For $c\in[0,1]$, let $(a_j(c))_{j=1}^\infty=a_1(c)a_2(c)\cdots a_j(c)\cdots$ be a sequence 
with $a_j(c)\in\{0,1\}$ such that
\[c=\sum_{j=1}^\infty \frac{a_j(c)}{2^j}.\]
Among the two possible representations of a dyadic rational except $0$,
we choose the one satisfying
\[\#\{j\in\mathbb N \colon a_j(c)=1\}=\infty,\]
so that the expansion is uniquely determined. In case $c=0,$ define $a_j(0)=0$ for any $j\ge 1.$
We call the sequence $(a_j(c))_{j=1}^\infty$ the binary expansion of $c.$
\end{df}
In what follows, for each finite word $\omega=\omega_1\cdots\omega_n\in \{0,1\}^n$ we set $\overline{\omega}=\omega_1\cdots\omega_n\omega_1\cdots\omega_n\omega_1\cdots\omega_n\cdots$. For example, $(a_j(3/8))_{j=1}^{\infty}=010\overline{1}$. Remark that if $c$ is a non-dyadic rational, $\#\{j\in \mathbb N\colon a_j(c)=0\}=\#\{j\in \mathbb N\colon a_j(c)=1\}=\infty.$

As in the case of IFSs, the limit set of a non-autonomous IFS admits a natural covering determined by the system (see Lemma~\ref{lem:Huchinson}).
This allows us to associate a family of nerves, denoted by $\mathcal{N}_{1, k}$ for $k > 1$ (see Definition~\ref{nerve}).
We then study the growth rate of the rank of the homology groups of $\mathcal{N}_{1, k}$ as $k \to \infty$, which reflect the number of connected components or holes in the slices.

\begin{syuB}
\label{SyuB}
Let $c\in [0, 1].$ Then the following holds.
\begin{itemize}
\item[(a)]Let $c$ be a dyadic rational with its binary expansion $a_1\cdots a_n0\overline{1}$ and $\ell=\#\{j\in \{1,\ldots, n\}\colon a_j=0\}.$ 
Then we have ${\rm rank}{\check H}_0(J_c)=3^{n-\ell}$ and 
\[\lim_{n\to \infty}\frac{1}{n}{\log {\rm rank}{H}_1(\mathcal N_{1, n+1})}=\log 3.\]
The same statements hold for cohomology. 
\item[(b)]If $c$ is a non-dyadic rational with its binary expansion $(a_j(c))_{j=1}^{\infty}$,  then for any $n\in \mathbb N$ we have \[\log {\rm rank}{H}_0(\mathcal N_{1, n+1}) = \log {\rm rank}{H}^0(\mathcal N_{1, n+1}) = \sum_{j=1}^n a_j(c) \log 3.\]
\end{itemize}
\end{syuB}
As consequences of Main Theorem~B(b), we obtain the following corollaries.
\begin{cor}
\label{Birk}
For Lebesgue almost every $c\in [0, 1]$ with $(a_j(c))_{i=1}^{\infty}$, we have \[\lim_{n\to \infty}\frac{1}{n}{\log {\rm rank}{H}_0(\mathcal N_{1, n+1})}=\frac{\log 3}{2} .\] 
\end{cor}
\begin{proof}
By Birkoff's ergodic theorem, for Lebesgue  almost every $c\in [0, 1],$ we have $\displaystyle{\lim_{n\to \infty}\frac{1}{n}\sum_{j=1}^n a_j(c)=\frac{1}{2}}$. Combining this with Main Theorem B yields the desired one. 
\end{proof}

\begin{cor}
\label{Dim}
Let $c$ be a non-dyadic rational. Then 

\begin{align*}
\dim_{\rm H}J_c=\underline{\dim}_{\rm B}J_c=\liminf_{n\to \infty}\frac{\log {\rm rank}{H}_0(\mathcal N_{1, n+1})}{n\log 2}
\end{align*}
 and 
\begin{align*}
\dim_{\rm P}J_c=\overline{\dim}_{\rm B}J_c=\limsup_{n\to \infty}\frac{\log {\rm rank}{H}_0(\mathcal N_{1, n+1})}{n\log 2},
\end{align*}
 where $\dim_{\rm H}, \underline{\dim}_{\rm B}, \dim_{\rm P},$ and $\overline{\dim}_{\rm B}$ denote the Hausdorff, lower box, packing, and upper box dimensions, respectively.
\end{cor}
\begin{proof}
Let $c$ be a non-dyadic rational with its binary expansion $(a_j(c))_{j=1}^{\infty}$.
By \cite[Main Theorem A]{Nak22}, we have 
\begin{align*}
\dim_{\rm H}J_c=\underline{\dim}_{\rm B}J_c=\liminf\limits_{n\rightarrow \infty}\frac{1}{n}\sum_{j=1}^{n}\frac{\log 3}{\log 2}a_j(c)
\end{align*}
 and 
\begin{align*}
\dim_{\rm P}J_c=\overline{\dim}_{\rm B}J_c=\limsup\limits_{n\rightarrow \infty}\frac{1}{n}\sum_{j=1}^{n}\frac{\log 3}{\log 2}a_j(c).
\end{align*}
Combining this with Main Theorem B yields the desired one. 
\end{proof}
The same statements as Corollaries~\ref{Birk} and \ref{Dim} hold for cohomology. 

\section{Non-autonomous iterated function systems and \\ \v{C}ech-Sumi homology group}
In this section, we summarize the theory of a homological framework in \cite{NW}. We first introduce the definition of a non-autonomous iterated function system (NIFS), which plays a central role in this paper.
\begin{df}\label{NIFS}

A non-autonomous iterated function system $(\Phi^{(j)})_{j=1}^{\infty}$ on a compact metric space $X$ is a sequence of collections   $\Phi^{(j)} = \{\varphi_{i}^{(j)}  \colon X \to X\}_{i \in I^{(j)}}$ of maps, where each index set $I^{(j)}$ is finite, and 
the Lipschitz constants of $\varphi_{i}^{(j)}$ are uniformly bounded above by a constant strictly less than $1$ for all $j \geq 1$ and $i \in I^{(j)}$. 

Let $(\Phi^{(j)})_{j=1}^{\infty}$ be a NIFS. Define the coding map $\Pi  \colon \prod_{j=1}^{\infty} I^{(j)}\to X$ of $(\Phi^{(j)})_{j=1}^{\infty}$ by \[\{\Pi(i_{1}, i_{2}, \dots)\}  = \bigcap_{j=1}^{\infty} \varphi^{(1)}_{i_{1}} \circ  \varphi^{(2)}_{i_{2}} \circ \dots  \circ \varphi^{(j)}_{i_{j}} (X),\] which is a singleton by the uniform contraction condition.
Define the limit set $J$ of $(\Phi^{(j)})_{j=1}^{\infty}$ by $J= \Pi(\prod_{j=1}^{\infty} I^{(j)}).$ 
\end{df}
Let $(\Phi^{(j)})_{j=1}^{\infty}$ be a NIFS. For every $j \geq 1$, consider a NIFS $(\Phi^{(j-1+k)})_{k=1}^{\infty}$ and its limit set $J_{j}$. 
Compared with usual IFSs, the limit set of a NIFS does not exhibit self-similarity, but it possesses the following nested structure.

\begin{lem}\cite[Lemma 3.3]{NW}\label{lem:Huchinson}
We have $J = J_{1}$, and for every $j < k$ we have 
\[J_{j} = \bigcup_{(i_{j}, \dots, i_{k-1}) \in I^{(j)} \times \dots \times  I^{(k-1)}} 
     \varphi^{(j)}_{i_{j}} \circ  \dots  \circ \varphi^{(k-1)}_{i_{k-1}}(J_{k}).\]
\end{lem}
By Lemma~\ref{lem:Huchinson} the collection $\{\varphi^{(j)}_{i_{j}} \circ  \dots  \circ \varphi^{(k-1)}_{i_{k-1}}(J_{k})\}_{(i_{j}, \dots, i_{k-1}) \in I^{(j)} \times \dots \times  I^{(k-1)}}$ is a covering 
 of $J_j$ for every $1 \leq j < k$. 
This enables us to define a simplicial complex.
\begin{df}\label{nerve}
Let $(\Phi^{(j)})_{j=1}^{\infty}$ be a NIFS.
For $1\le j<k$, define $\mathcal N_{j,k}$ as the nerve for the covering of Lemma~\ref{lem:Huchinson}. Namely,  $\mathcal N_{j,k}$ is the simplicial complex defined as follows.
Its vertex set is $I^{(j)} \times \cdots \times I^{(k-1)}$.
For each vertex $v=(i_j,\ldots,i_{k-1})$, set
$\varphi_v := \varphi_{i_j}^{(j)} \circ \cdots \circ \varphi_{i_{k-1}}^{(k-1)}.$
A finite set $\{v_0,\ldots,v_q\}$ of distinct vertices forms a simplex of $\mathcal N_{j,k}$ if and only if
\[\bigcap_{p=0}^q \varphi_{v_p}(J_k)\neq\emptyset.\]
\end{df}
Define a simplicial map $\phi_{j, k}  \colon \mathcal{N}_{j, k+1} \to \mathcal{N}_{j, k}$ 
so that \[\phi_{j, k}(i_{j}, \dots,  i_{k-1}, i_{k}) = (i_{j}, \dots,  i_{k-1}),\] which induces a homomorphism 
\[(\phi_{j, k})_{*}  \colon H_{q}(\mathcal{N}_{j, k+1}) \to H_{q}(\mathcal{N}_{j, k})\] 
on the homology groups (with $\ZZ$ coefficients) for every $q \geq 0$. 
Define the $q$th \v{C}ech-Sumi homology group $\varprojlim_{k}{H}_q(\mathcal{N}_{1, k})$ of the NIFS $(\Phi^{(j)})_{j=1}^{\infty}$ as the inverse limit of the inverse system $\{(\phi_{1, k})_{*}  \colon H_{q}(\mathcal{N}_{1, k+1}) \to H_{q}(\mathcal{N}_{1, k})\}_{k = 2}^{\infty}$.
Similarly, define the $q$th \v{C}ech-Sumi cohomology group $\varinjlim_{k}{H}^q(\mathcal{N}_{1, k})$ as the direct limit of the direct system $\{(\phi_{1, k})^{*}  \colon H^{q}(\mathcal{N}_{1, k}) \to H^{q}(\mathcal{N}_{1, k+1})\}_{k = 2}^{\infty}$.
Then the following holds.
\begin{thm}\cite[Theorem 4.6]{NW}
\label{th:cechsumi}
There is an isomorphism between 
the \v{C}ech-Sumi homology group $\varprojlim_{k}{H}_q(\mathcal{N}_{1, k})$ and 
the \v{C}ech homology group $\check{H}_q(J)$ for every $q \geq 0$. 
Also, there is an isomorphism between 
the \v{C}ech-Sumi cohomology group $\varinjlim_{k}{H}_q(\mathcal{N}_{1, k})$ and 
the \v{C}ech cohomology group $\check{H}^q(J)$ for every $q \geq 0$. 
\end{thm}
\section{Proof}
In this section we give proofs of main results.
Recall that $v_0=(0, 0, 0), v_1=(0, 0, 1), v_2=(0, 1, 1),$ $v_3=(1, 1, 1)$ and for each $i \in \{0,1,2,3\}$, a map $f_i: \mathbb R^3\rightarrow \mathbb R^3$ is defined by $f_i(x)=(x+v_i)/2$.
\subsection{Preliminary lemmas}
Let $c\in [0, 1]$ and let $(a_j(c))_{j=1}^{\infty}$ be its binary  expansion defined in Definition~\ref{bi}. For each $j\in \mathbb N$ define \[A_c^{(j)}=
 \begin{cases}
 \{0\}&(a_j(c)=0)\\
 \{1,2,3\}&(a_j(c)=1),
 \end{cases}
 \]
 and consider a non-autonomous IFS $\Phi_c=(\Phi_c^{(j)})_{j=1}^{\infty}$ which is defined by $\Phi_c^{(j)}=\{f_i\}_{i\in A_c^{(j)}},\ j\geq 1.$  Let $\Pi_{\Phi_c}$ be the coding map of $\Phi_c.$ 
 \begin{lem}
\label{coding}
For any $\omega=\omega_1\omega_2\cdots\in \{0,1,2,3\}^{\infty},$ we have 
$\displaystyle{\Pi_{\Phi_c}(\omega)=\sum_{j=1}^{\infty}\frac{v_{\omega_j}}{2^j}.}$
\end{lem}
\begin{proof}
Take $\omega=\omega_1\omega_2\cdots\in \{0,1,2,3\}^{\infty}$. We first prove by induction on $n$ that for any $n\ge 1$,
\begin{equation}\label{compo}
f_{\omega|_n}(x)=\frac{x}{2^n}+\sum_{j=1}^n \frac{v_{\omega_j}}{2^j}\ \text{for any}\ x\in\mathbb{R}^3,
\end{equation}
where $\omega|_n=\omega_1\cdots\omega_n.$
For $n=1$, since $f_{\omega_1}(x)=(x+v_{\omega_1})/2$, we have
\[
f_{\omega|_1}(x)=\frac{x}{2}+\frac{v_{\omega_1}}{2},
\]
and \eqref{compo} holds.

Assume \eqref{compo} holds for some $n\ge 1$.
Then
\[f_{\omega|_{n+1}}(x)
=f_{\omega|_n}\bigl(f_{\omega_{n+1}}(x)\bigr)
=f_{\omega|_n}\Bigl(\frac{x+v_{\omega_{n+1}}}{2}\Bigr).\]
Applying the induction hypothesis with $y=(x+v_{\omega_{n+1}})/2$ gives
\[f_{\omega|_{n+1}}(x)
=\frac{1}{2^n} \frac{x+v_{\omega_{n+1}}}{2}
+\sum_{j=1}^n \frac{v_{\omega_j}}{2^j}
=\frac{x}{2^{n+1}}+\sum_{j=1}^n \frac{v_{\omega_j}}{2^j}+\frac{v_{\omega_{n+1}}}{2^{n+1}},\]
which is \eqref{compo} for $n+1$.
Hence, \eqref{compo} holds for every $n$.

By the definition of the coding map,
\[\Pi_{\Phi_c}(\omega)=\lim_{n\to\infty} f_{\omega|_n}(x_0),\]
for any choice of base point $x_0\in\mathbb{R}^3$.
Taking $x_0=0$ in \eqref{compo} yields
\[f_{\omega|_n}(0)=\sum_{j=1}^n \frac{v_{\omega_j}}{2^j}.\]
Letting $n\to\infty$ gives
\[\Pi_{\Phi_c}(\omega)=\sum_{j=1}^\infty\frac{v_{\omega_j}}{2^j},\]
which completes the proof.
\end{proof}
The following lemma states that each slice of the Sierpi\'nski tetrahedron is regarded as the limit set of a NIFS.
\begin{lem}
\label{SliceNIFS}
For any $c\in [0, 1]$ the set $J_c$ is the limit set of the NIFS $\Phi_c.$
In particular,
if $c$ is a dyadic rational with $a(c)=a_1\cdots a_n0\overline{1}$, then 
\[
J_c=\bigcup_{\omega_1\cdots \omega_{n+1}\in \prod_{j=1}^n A_c^{(j)}\times \{0\}}f_{\omega_1\cdots \omega_{n+1}}({\rm SG}),\]where ${\rm SG}$ denotes the limit set of the IFS $\{f_1, f_2, f_3\}.$

\end{lem}
\begin{proof}
Although this lemma follows immediately from \cite{Nak22},
we include the argument for completeness.
By Lemma~\ref{coding} we have
\begin{equation*}
\begin{split}
J_c&=\{x\in J\colon p_z(x)=c\}\\&=\{\Pi_{\Phi_c}(\omega)\colon p_z(\Pi_{\Phi_c}(\omega))=c\ \mbox{for}\ \omega\in \{0, 1,2,3\}^{\infty}\}\\&=\left\{\Pi_{\Phi_c}(\omega)\ :\ p_z\left(\sum_{j=1}^{\infty}(1/2)^{j}v_{\omega_{j}}\right)=c\ \mbox{for}\ \omega\in \{0,1,2,3\}^{\infty}\right\}\\&=\left\{\Pi_{\Phi_c}(\omega)\colon \sum_{j=1}^{\infty}(1/2)^{j}p_z(v_{\omega_{j}})=c\ \mbox{for}\ \omega\in \{0,1,2,3\}^{\infty}\right\}.
\end{split}
\end{equation*}
If $c$ is a non-dyadic rational with its binary expansion $(a_j(c))$, 
by $z_j:=p_z(v_{\omega_{j}})\in\{0,1\}$ and the uniqueness of the binary expansion, we have that $z_j=a_j(c)$ for any $j\in \mathbb N$. Hence, $J_c$ coincides with \[\left\{\Pi_{\Phi_c}(\omega)\colon \omega_j\in A_c^{(j)}\ \text{for any}\ j\in \mathbb N\right\},\] which is the limit set of the NIFS $\Phi_c.$

Let $c\in [0,1]$ be a dyadic rational with its binary expansion $a_1\cdots a_n0\overline{1}$.
Then we have 
\begin{align*}
J_c=\left\{\Pi_{\Phi_c}(\omega)\colon \sum_{j=1}^{\infty}(1/2)^{j}p_z(v_{\omega_{j}})=c\ \mbox{for}\ \omega\in \{0,1,2,3\}^{\infty}\right\}.
\end{align*}
Since $z_j:=p_z(v_{\omega_{j}})\in\{0,1\}$ we have that $z_1z_2\cdots z_n\cdots=a_1a_2\cdots a_n1\overline{0}$ or $z_1z_2\cdots z_n\cdots=a_1a_2\cdots a_n0\overline{1}$.
Hence, we have 
\[
J_c=\bigcup_{\omega_1\cdots \omega_{n+1}\in \prod_{j=1}^n A_c^{(j)}\times \{1, 2, 3\}}f_{\omega_1\cdots \omega_{n+1}}(\boldsymbol{o})\cup \bigcup_{\omega_1\cdots \omega_{n+1}\in \prod_{j=1}^n A_c^{(j)}\times \{0\}}f_{\omega_1\cdots \omega_{n+1}}({\rm SG}),\]
where $\boldsymbol{o}$ denotes the origin $(0,0,0)$ and ${\rm SG}$ denotes the limit set of IFS $\{f_1, f_2, f_3\}$. 
Take $\omega_1\cdots\omega_{n+1}\in \prod_{j=1}^n A_c^{(j)}\times \{1, 2, 3\}$ and set $f_{\omega_{n+1}}(x)=(x+v_{\omega_{n+1}})/2.$ Since $f_{\omega_{n+1}}(\boldsymbol{o})=v_{\omega_{n+1}}/2$ and $v_{\omega_{n+1}}$ is a fixed point of $f_{\omega_{n+1}}$ belonging to the limit set ${\rm SG}$ of $\{f_1, f_2, f_3\}$, we have $f_{\omega_1\cdots \omega_{n+1}}(\boldsymbol{o})=f_{\omega_1}\circ\cdots\circ f_{\omega_n}(v_{\omega_{n+1}}/2)=f_{\omega_1}\circ\cdots\circ f_{\omega_n}(f_0(v_{\omega_{n+1}}))\in f_{\omega_1}\circ\cdots\circ f_{\omega_n}\circ f_0({\rm SG}).$ Hence, we have 
\[
J_c=\bigcup_{\omega_1\cdots \omega_{n+1}\in \prod_{j=1}^n A_c^{(j)}\times \{0\}}f_{\omega_1\cdots \omega_{n+1}}({\rm SG}),\]which is the limit set of the NIFS $(\Phi_c^{(j)})_{j=1}^{\infty}.$
\end{proof}
Let $X=[0, 1]^2$ be the unit square in $\mathbb R^2$. Let $b_1=(0, 0), b_2=(0, 1),$ and $b_3=(1, 1)$. 
For each $i \in \{1,2,3\}$, define a contracting map $\psi_i: X\rightarrow X$ by $\psi_i(x)=(x+b_i)/2$. 

Let $c\in [0, 1]$ and let $(a_j(c))_{j=1}^{\infty}$ be its binary expansion defined in Definition~\ref{bi}. For each $j\in \mathbb N$ define \[I_c^{(j)}:=
 \begin{cases}
 \{1\}&(a_j(c)=0)\\
 \{1, 2, 3\}&(a_j(c)=1)
 \end{cases}
 \]
 and consider a NIFS $\Psi_c=(\Psi_c^{(j)})_{j=1}^{\infty}, \Psi_c^{(j)}=\{\psi_{i}\}_{i\in I^{(j)}}, j\geq 1.$
 Let $\Pi_{\Psi_c}$ be the coding map of $\Psi_c.$ 
\begin{lem}
\label{coding2}
For any $\omega=\omega_1\omega_2\cdots\in \{1,2,3\}^{\infty},$ we have 
$\displaystyle{\Pi_{\Psi_c}(\omega)=\sum_{j=1}^{\infty}\frac{b_{\omega_j}}{2^j}.}$
\end{lem}
\begin{proof}
The proof is similar to the proof of Lemma~\ref{coding}.
\end{proof}
Let $P$ be the projection from $\mathbb R^3$ onto $\mathbb R^2$ defined by $P(x_1,x_2,x_3)=(x_1, x_2).$ 
\begin{lem}
\label{Proj}
For any $c\in [0, 1]$, the map $P|_{J_c}$ is a homeomorphism onto the limit set of the NIFS $\Psi_c.$
\end{lem}
\begin{proof}
By Lemma~\ref{SliceNIFS} for any $c\in [0, 1]$ the set $J_c$ is the limit set of the NIFS $(\Phi_c^{(j)})_{j=1}^{\infty}.$ 
By Lemma~\ref{coding} and Lemma~\ref{coding2} we have
\begin{equation*}
\begin{split}
P(J_c)&=P\left(\left\{\sum_{j=1}^{\infty}\frac{v_{\omega_j}}{2^j}\colon \omega_j\in A_c^{(j)}\ \text{for all}\ j\geq 1\right\}\right)\\
&=\left\{\sum_{j=1}^{\infty}\frac{P(v_{\omega_j})}{2^j}\colon \omega_j\in A_c^{(j)}\ \text{for all}\ j\geq 1\right\}=\left\{\sum_{j=1}^{\infty}\frac{b_{\omega_j}}{2^j}\colon \omega_j\in I_c^{(j)}\ \text{for all}\ j\geq 1\right\},
\end{split}
\end{equation*}
which is the limit set of $\Psi_c. $
\end{proof}
 By Lemma~\ref{Proj}, for each $c$ we have
$P(J_c)=\{x-(0,0,c)\in \mathbb R^3\colon x\in J_c\}.$
Hence, in order to study the slice $J_c$ at height $c$, it suffices to investigate its projection $P(J_c)$.
In what follows, we simply write $J_c$ and $\Pi_c$ instead of $P(J_c)$ and $\Pi_{\Psi_c}$, respectively.

\begin{lem}
\label{Split}
Let $1\leq j< k$. Then for any distinct $\omega=\omega_j\cdots \omega_{k-1}, \tau=\tau_j\cdots \tau_{k-1}\in \{1, 2, 3\}^{k-j}$, we have $\psi_{\omega1}(X)\cap \psi_{\tau1}(X)=\emptyset.$
\end{lem}
\begin{proof}
For each word $\omega\in\{1,2,3\}^{k-j}$, the set $\psi_\omega(X)$ is a square of side length $2^{-(k-j)}$.
Then we have
\[\psi_{\omega 1}(X)=\psi_\omega(\psi_1(X))=\psi_\omega\bigl([0,1/2]^2\bigr),\]
which is exactly the lower-left sub-square obtained by dividing
$\psi_\omega(X)$ into four congruent squares.
Hence, it has side length $2^{-(k-j+1)}$
and lies in the lower-left quarter of $\psi_\omega(X)$.

Since $\omega\neq \tau$, squares $\psi_\omega(X)$ and $\psi_\tau(X)$ are distinct and their interiors are disjoint. In particular, their lower-left quarters
are disjoint. Therefore, $\psi_{\omega1}(X)\cap \psi_{\tau1}(X)=\emptyset.$
\end{proof}
\begin{pro}
\label{Dissconne}
If $c$ is a non-dyadic rational, then $J_c$ is totally disconnected.
\end{pro}
\begin{proof}
Since $c$ is not a dyadic rational, there exists a strictly increasing sequence $(n_i)_{i=1}^{\infty}$ such that $I^{(n_i)}_c=\{1\}$ for any $i\geq 1.$  Take $x\in J_c$, let $C_x$ denote the connected component containing $x$ and let $\omega=\omega_1\omega_2\cdots \in \prod_{j=1}^{\infty}I^{(j)}$ satisfy $\Pi_c(\omega)=x.$ By Lemma~\ref{Split}, $C_x\subset \psi_{\omega_1\cdots \omega_{n_i}}(X)$ for any $i\geq 1.$ Since $|C_x|\leq |\psi_{\omega_1\cdots \omega_{n_i}}(X)|\le 1/2^{n_i}$ and letting $i\to \infty$ we obtain $|C_x|=0,$ which completes the proof.
\end{proof}
In what follows, given the NIFS $\Psi_c$ and $n \in \mathbb{N}$, we denote by $\mathcal{N}_{1,n+1}$ the simplicial complex introduced in Definition~\ref{nerve}.
\begin{lem}
\label{ConnNum}
 Let $1\leq m< n$ and $\ell=\#\{i\in \{m,\ldots, n-1\}\colon I_c^{(i)}=\{1\}\}.$ If $I_c^{(n)}=\{1\}$, then ${H}_0(\mathcal{N}_{m,n+1})$ and ${H}^0(\mathcal{N}_{m,n+1})$ are the free abelian group of rank $3^{n-m-\ell}.$

\end{lem}
\begin{proof}

Let $1\leq m< n$ and $\ell=\#\{i\in \{m,\ldots, n-1\}\colon I^{(i)}=\{1\}\}.$
By Lemma~\ref{Split}, there is no 1-simplex on  $\mathcal{N}_{m,n+1}$, which implies the desired one.
\end{proof}

\begin{pro}
\label{DyaConnNum}
Let $c$ be a dyadic rational with its binary expansion $a_1\cdots a_n0\overline{1}$ and $\ell=\#\{i\in \{1,\ldots, n\}\colon a_i(c)=0\}.$ Then $\check{H}_0(J_c)$ and $\check{H}^0(J_c)$ are the free abelian group of rank $3^{n-\ell}.$
\end{pro}
\begin{proof}
By Lemmas~\ref{SliceNIFS} and ~\ref{Split} we have \[
J_c=\bigcup_{\omega_1\cdots \omega_{n+1}\in \prod_{j=1}^n I_c^{(j)}\times \{0\}}\psi_{\omega_1\cdots \omega_{n+1}}({\rm SG})\ \text{(disjoint union)}, \]where SG is the limit set of $\{\psi_1, \psi_2, \psi_3\}.$
Since SG is connected, we have the desired one by Lemma~\ref{ConnNum}.

\end{proof}
\begin{pro}
\label{NotDyaConnNum}
If $c$ is a non-dyadic rational with its binary expansion $(a_j(c))_{i=1}^{\infty}$,  then we have \[\log {\rm rank}{H}_0(\mathcal N_{1, n+1})=\log {\rm rank}{H}^0(\mathcal N_{1, n+1})=\sum_{j=1}^n a_j(c)\log 3 .\] 
\end{pro}

\begin{proof}
Let $n\in \mathbb N.$ Since $c$ is not a dyadic rational, there exists $j_0> n$ such that $a_{j_0}(c)=0,$ which means $I^{(j_0)}_c=\{1\}.$ Take $\omega\neq \tau\in I^{(1)}\times \cdots \times I^{(n)}$ and let $K_n$ be the limit set of the NIFS $(\Psi_c^{(n+k)})_{k=1}^{\infty}.$ By Lemma~\ref{Split} for any $\mu\in I_c^{(n+1)}\times \cdots \times I_c^{(j_0-1)},$ we have 
\begin{equation}
\label{equ1}
\begin{split}\psi_{\omega}\circ \psi_{\mu}\circ \psi_1(X)\cap \psi_{\tau}\circ \psi_{\mu}\circ \psi_1(X)=\emptyset.
\end{split}
\end{equation}
By the definition of $K_n$, we have 
\begin{equation}
\label{equ2}
\begin{split}
K_n\subset \bigcup_{\mu\in I_c^{(n+1)}\times \cdots \times I_c^{(j_0-1)}}\psi_{\mu}\circ \psi_1(X).
\end{split}
\end{equation}
Combining (\ref{equ1}) and (\ref{equ2}), we have 
$\psi_{\omega}(K_n)\cap \psi_{\tau}(K_n)=\emptyset.$ 
Hence, there is no 1-simplex on  $\mathcal{N}_{1,n+1}$, which implies that \[{\rm rank}{H}_0(\mathcal N_{1, n+1})={\rm rank}{H}^0(\mathcal N_{1, n+1})=\# (I^{(1)}_c\times\cdots \times I^{(n)}_c).\]
By the definition of $I^{(j)}_c$ we have \[\log {\rm rank}{H}_0(\mathcal N_{1, n+1})=\log {\rm rank}H^0(\mathcal N_{1, n+1})=\sum_{j=1}^n a_j(c)\log 3 ,\] which completes the proof.
%
%
%
%
\end{proof}

\subsection{Proofs of main results}
By \cite[Theorem 1.9 (2)]{Sumi09}, we have 
$\check{H}_0(\mathrm{SG}) \cong  \check{H}^0(\mathrm{SG}) \cong \mathbb{Z}$,  
$\mathrm{rank} \check{H}_1(\mathrm{SG}) = \mathrm{rank} \check{H}^1(\mathrm{SG}) = \infty$,
and  $ \check{H}_q(\mathrm{SG}) = \check{H}^q(\mathrm{SG})= 0$ for all $q \geq 2$. Moreover, the growth rate of the rank of the first (co)homology of $\mathrm{SG}$ is equal to $\log 3$. 

If $c$ is a non-dyadic rational, by Lemma~\ref{Split} there exists a strictly increasing sequence $(n_i)_{i=1}^{\infty}$ such that ${H}_1(\mathcal N_{1, n_i+1})  = {H}^1(\mathcal N_{1, n_i+1}) =0.$ Combining this with Theorem~\ref{th:cechsumi} yields $\check{H}_1(J_c) = \check{H}^1(J_c)=0$
and Proposition~\ref{Dissconne} implies Main Theorem A (b). 
Proposition~\ref{DyaConnNum} implies Main Theorem A (a) and Main Theorem B~(a). 
Finally, by Proposition~\ref{NotDyaConnNum} we obtain Main Theorem B (b), which completes the proof.

\section{The general dimensional setting}
Let $d\ge 3$. As in \cite{Nak22}, we introduce the $d$-dimensional Sierpi\'nski gasket, where the $3$-dimensional Sierpi\'nski gasket coincides with the Sierpi\'nski tetrahedron.
\begin{df}
\label{dsier}
Let $v_0=(0,0,...,0)\in \mathbb{R}^d$. Let $v_i=v_0+\sum_{j=d-i+1}^de_j\in \mathbb{R}^d$ for each $i\in \{1, 2,..., d\}$, where $e_j$ is the $j$th element of the natural basis of $\mathbb{R}^d$. Let $V_d=\{v_0, v_1,..., v_{d}\},$ i.e.

\begin{align*}
V_d:=\bigl\{&(0, 0, 0,..., 0, 0, 0),(0, 0, 0,..., 0, 0, 1),(0, 0, 0,..., 0, 1, 1),\cdots,\\&(0, 0, 1,..., 1, 1, 1),(0, 1, 1,..., 1, 1, 1),(1, 1, 1,..., 1, 1, 1)\}\bigr .
\end{align*} 
For each $v \in V_d$, define a contracting map $g_v: [0, 1]^d\rightarrow [0, 1]^d$ by $g_v(x)=(x+v)/2$. Then the set of $d+1$ maps $\{g_v\}_{v\in V_d}$ constructs an IFS. Then there uniquely exists a non-empty compact subset $J^d$ such that 

\begin{align*}
J^d=\bigcup_{v\in V_d}g_v(J^d),
\end{align*}
which is called the $d$-dimensional Sierpi\'nski gasket.
\end{df}
Let $x=(x_1,x_2,...,x_d)\in \mathbb{R}^d$. We set $p_d(x)=x_d$. Let $c\in \mathbb{R}$ and consider the following level set $J^d_{c}$ of $J^d$. 

\begin{align*}
J^d_{c}=\{x\in J^d\colon p_d(x)=c\}.
\end{align*}
Arguing as in the case $d=3$, we obtain the following results analogous to Main Theorems~A and~B.
\begin{thm}
\label{thm1}
Let $c\in [0, 1].$ Then the following holds.
\begin{itemize}
\item[(a)]If $c$ is a dyadic rational, then $J^d_c$ is a finite disjoint union of copies of the $(d-1)$-dimensional Sierpi\'nski gasket, $0<{\rm rank}\check{H}_0(J^d_c)<\infty,\ {\rm rank}\check{H}_1(J^d_c)=\infty,$ and $\check{H}_q(J^d_c)=0$ for all $q\geq 2.$
\item[(b)]If $c$ is a non-dyadic rational, then $J^d_c$ is totally disconnected and $\check{H}_q(J^d_c)=0$ for all $q\geq 1.$
\end{itemize}

\end{thm}
\begin{thm}
\label{thm2}
Let $c\in [0, 1].$ Then the following holds.
\begin{itemize}
\item[(a)]Let $c$ be a dyadic rational with its binary expansion $a_1\cdots a_n0\overline{1}$ and $\ell=\#\{j\in \{1,\ldots, n\}\colon a_j=0\}.$ Then we have ${\rm rank}{\check H}_0(J^d_c)=d^{n-\ell}$ and 
\[\lim_{n\to \infty}\frac{1}{n}{\log {\rm rank}{H}_1(\mathcal N_{1, n+1})}=\log d.\]
\item[(b)]If $c$ is a non-dyadic rational with its binary expansion $(a_j(c))_{j=1}^{\infty}$,  then for any $n\in \mathbb N$ we have \[\log {\rm rank}{H}_0(\mathcal N_{1, n+1})=\sum_{j=1}^n a_j(c)\log d.\]
\end{itemize}

The same holds for cohomology groups.

\end{thm}
\begin{rem}
By the Alexander duality theorem \cite[Theorem 6.2.16]{Spa}, 
the topology of a subset $D \subset \mathbb{R}^d$ determines that of its complement.
More precisely, 
the \v{C}ech cohomology group $\check{H}^{d-q-1}(D)$ is isomorphic to the reduced homology group $\tilde{H}_{q}(\RR^d\setminus D)$ of the complement. 
In our setting, we deduce the following results from Main Theorems~A and~B.
\begin{thm}
\label{thm3}
Let $c\in [0, 1].$ Then the following holds.
\begin{itemize}
\item[(a)]Let $c$ be a dyadic rational with its binary expansion $a_1\cdots a_n0\overline{1}$ for some $n\in \mathbb N$ and $\ell=\#\{j\in \{1,\ldots, n\}\colon a_j=0\}.$ Then ${\rm rank}\tilde{H}_{d-1}(\mathbb R^d\setminus J_c)=d^{n-\ell},\ {\rm rank}{\tilde H}_{d-2}(\mathbb R^d\setminus J_c)=\infty$, and $\tilde{H}_q(\mathbb R^d\setminus J_c)=0$ for all $q\notin \{d-1, d-2\}.$
\item[(b)]If $c$ is a non-dyadic rational, then ${\rm rank}\tilde{H}_{d-1}(\mathbb R^d\setminus J_c)=\infty$ and ${\rm rank}\tilde{H}_q(\mathbb R^d\setminus J_c)=0$ for all $q\neq d-1.$
\end{itemize}

\end{thm}
\end{rem}
\subsection*{Acknowledgments} 
YN is partially supported by  JSPS KAKENHI JP25K17282. TW is partially supported by  JSPS KAKENHI
(JP23K13000, JP24K00526, JP25K00011) 
and by JST AIP Accelerated Program JPMJCR25U6.

\end{document}